\documentclass[12pt]{amsart}

 \usepackage{amsfonts,graphics,amsmath,amsthm,amsfonts,amscd,amssymb,amsmath,latexsym,multicol,
 mathrsfs}
\usepackage{epsfig,url}
\usepackage{flafter}
\usepackage{fancyhdr}
\usepackage{hyperref}
\hypersetup{colorlinks=true, linkcolor=black}

\makeatletter

\def\jobis#1{FF\fi
  \def\predicate{#1}%
  \edef\predicate{\expandafter\strip@prefix\meaning\predicate}%
  \edef\job{\jobname}%
  \ifx\job\predicate
}

\makeatother

\if\jobis{proposal}%
\else
\fi

 \usepackage[matrix, arrow]{xy}

\DeclareMathOperator{\Supp}{Supp}

\DeclareMathOperator{\codim}{codim}

 \numberwithin{equation}{subsection}
 \numberwithin{footnote}{subsection}

 \newtheorem{cor}[subsection]{Corollary}
 \newtheorem{lem}[subsection]{Lemma}
 \newtheorem{prop}[subsection]{Proposition}
 \newtheorem{thm}[subsection]{Theorem}
 \newtheorem{conj}[subsection]{Conjecture}

{
    \newtheoremstyle{upright}%
        {8pt plus2pt minus4pt}%
        {8pt plus2pt minus4pt}%
        {\upshape}%
        {}%
        {\bfseries\scshape}%
        {}%
        {1em}%
        {}%
\theoremstyle{upright}

 \newtheorem{exa}[subsection]{Example}

 \newtheorem{rem}[subsection]{Remark}

}

 \newcommand{\C}{\mathbb C}

 \newcommand{\Q}{\mathbb Q}
 \newcommand{\R}{\mathbb R}
 
 \newcommand{\bir}{\dashrightarrow}

\title{Log canonical pairs over varieties with maximal Albanese dimension}
\author{Zhengyu Hu}
\date{\today}
\begin{document}
\maketitle

\begin{abstract}
Let $(X,B)$ be a log canonical pair over a normal variety $Z$ with maximal Albanese dimension. If $K_X+B$ is relatively abundant over $Z$ (for example, 
$K_X+B$ is relatively big over $Z$), then we prove that  $K_X+B$ is abundant. In particular, the subadditvity of Kodaira dimensions $\kappa(K_X+B) \geq \kappa(K_F+B_F)+ \kappa(Z)$ holds, where $F$ is a general fiber, $K_F+B_F= (K_X+B)|_F$, and $\kappa(Z)$ means the Kodaira dimension of a smooth model of $Z$. We discuss several variants of this result in Section \ref{variants}.
We also give a remark on the log Iitaka conjecture for log canonical pairs in Section \ref{rem-conj}.
\end{abstract}



\section{Introduction}

Let $f: X \rightarrow Z$ be a surjective morphism from a smooth variety to a normal variety $Z$. An important conjecture in birational geometry is the \emph{Iitaka Conjecture} which asserts that 
$$
\kappa(X) \geq \kappa(F) +\kappa(Z)
$$
where $\kappa(X)$ is the Kodaira dimension of $X$, $F$ is a general fiber of $f$ and $\kappa(Z)$ means the Kodaira dimension of a smooth model of $Z$. This has been established in many cases [\ref{Kawamata-2}], [\ref{CH}], [\ref{BCZ}], [\ref{CZ}], [\ref{Fujino1}], [\ref{Fujino3}], Koll\'{a}r, Viehweg, etc.

It is natural to raise a similar conjecture of the \emph{log version}. Consider a surjective morphism $f: (X,B) \rightarrow Z$ from a log canonical (lc for short) pair to a normal variety $Z$. One conjectures that 
$$
\kappa(K_X+B) \geq \kappa(K_F+B_F) +\kappa(Z)
$$
where $K_F+B_F= (K_X+B)|_F$. This has also been established in many cases. Recently Cao and P\v{a}un [\ref{CP}], Hacon, Popa and Schnell [\ref{HPS}] proved this when $(X,B)$ is Kawamata log terminal (klt for short) and $Z=A$ is an abelian variety over $\C$ and Cao [\ref{Cao}] proved this when $(X,B)$ is klt and $Z$ is a complex surface. If $(X,B)$ is klt with an extra assumption of positivity on $K_X+B$ and $Z$ has maximal Albanese dimension, then Birkar and Chen [\ref{BC}] obtained a stronger result which further asserts that $(X,B)$ has a good log minimal model.

There are some other related conjectures for logarithmic Kodaira dimensions for algebraic fibrations. See Kawamata [\ref{Kawamata-3}], Fujino [\ref{Fujino2}], Iitaka [\ref{Iitaka}], etc.

In this paper we will mainly discuss the case when $(X,B)$ is log canonical and $Z$ has maximal Albanese dimension. We work over $\mathbb{C}$ throughout this paper. All varieties are quasi-projective and a divisor means a $\Q$-divisor unless stated otherwise.

\textbf{Abundance and good minimal models.}
The first main result of this paper is the following theorem.

\begin{thm}[=Theorem \ref{m-thm'}]\label{m-thm}
Let $f:(X,B) \rightarrow Z$ be a surjective morphism from a projective log canonical (lc for short) pair $(X,B)$ to a normal variety $Z$. Assume that $Z$ has maximal Albanese dimension and that $K_X+B$ is relatively abundant over $Z$. Then, $K_X+B$ is abundant. Moreover, the subadditvity of Kodaira dimensions $\kappa(K_X+B) \geq \kappa(K_F+B_F)+ \kappa(Z)$ holds, where $F$ is a general fiber, $K_F+B_F= (K_X+B)|_F$, and $\kappa(Z)$ means the Kodaira dimension of a smooth model of $Z$. 
\end{thm}

As an immediate corollary of Theorem \ref{m-thm} we obtain the following.

\begin{cor}\label{m-cor}
Let $f:(X,B) \rightarrow A$ be a morphism from a projective lc pair $(X,B)$ to an abelian variety $A$. Assume that $(X,B)$ relatively has a good minimal model over $A$. Then, $(X,B)$ has a good minimal model. In particular, if $K_X+B$ is semi-ample$/A$, then it is semi-ample.
\end{cor}

This is a generalization of [\ref{BC}, Theorem 1.1]. Here we sketch the main strategy to prove Theorem \ref{m-thm}. Since Birkar and Chen [\ref{BC}] already proved case when $(X,B)$ is klt such that $K_X+B$ is big$/Z$, we simply need to reduce to this easier case. To this end we apply a canonical bundle formula to get a generalized lc generalized pair $(Y,B_Y+M_Y)$ (see Section \ref{preliminaries} Preliminaries for definitions, [\ref{BZ}] for more details) such that $K_Y+B_Y+M_Y$ is big$/Z$. Thanks to the assumption $K_X+B$ being relatively abundant$/Z$, we deduce that $M_Y$ is nef and abundant by a perturbation of the coefficients of $B$. Again we modify the coefficients of $B_Y$ to get a desired klt pair $(Y,B_Y+M_Y)$. 

We also give a generalization of [\ref{Lai}, Theorem 4.2 and Corollary 4.3] as below.

\begin{cor}\label{s-cor}
Let $f:(X,B) \rightarrow A$ be the Albanese map from a projective klt pair $(X,B)$ to its Albanese variety $A=\mathrm{Alb}(X)$. Suppose that $(F,B_F)$ has a good minimal model where $F$ is the general fiber and $K_F+B_F=(K_X+B)|_F$. Then, $(X,B)$ has a good minimal model.
\end{cor}

\begin{cor}\label{s-cor-2}
	Let $f:(X,B) \rightarrow Z$ be a surjective morphism from a projective klt pair $(X,B)$ to a normal variety $Z$. Suppose that $Z$ has maximal Albanese dimension, and that $(F,B_F)$ has a good minimal model where $F$ is the general fiber and $K_F+B_F=(K_X+B)|_F$. Then, $(X,B)$ has a good minimal model.
\end{cor}

In fact one can prove a stronger version on lc pairs instead of klt pairs as follows.

\begin{prop}[=Proposition \ref{m-prop'}]\label{m-prop}
	Let $f:(X,B) \rightarrow Z$ be a surjective morphism from a projective lc pair $(X,B)$ to a normal variety $Z$. Suppose that $Z$ has maximal Albanese dimension, and that $(F,B_F)$ has a good minimal model where $F$ is the general fiber and $K_F+B_F=(K_X+B)|_F$. Then, $K_X+B$ is abundant. Moreover, if the canonical ring $\mathcal{R}(K_X+B)$ is finitely generated, for example, $(X,B)$ is klt, or $\kappa(K_X+B)=0$, then $(X,B)$ has a good minimal model. 
\end{prop}

\textbf{Generalized polarized pairs and weak nonvanishing.}
It is known that minimal model theory works for generalized lc generalized pairs, although the abundance and the nonvanishing are not expected in general. However, there are interesting cases that abundance holds.

\begin{prop}[=Proposition \ref{prop2}]
	Let $(X,B+M)$ be a generalized lc generalized pair with data $X' \to X$ and $M'$, and $f:(X,B+M) \rightarrow Z$ be a surjective morphism to a normal variety $Z$. Assume that 
	
	$\bullet$ $Z$ has maximal Albanese dimension,
	
	$\bullet$ $K_X+B+M$ is abundant$/Z$,
	
	$\bullet$ $M'$ is abundant, and
	
	$\bullet$ every lc center of $(X,B+M)$ is vertical$/Z$. 
	
	Then, $K_X+B+M$ is abundant. Moreover, the subadditvity of Kodaira dimensions $\kappa(K_X+B+M) \geq \kappa(K_F+B_F+M_F)+ \kappa(Z)$ holds, where $F$ is a general fiber, $K_F+B_F+M_F= (K_X+B+M)|_F$, and $\kappa(Z)$ means the Kodaira dimension of a smooth model of $Z$. 
\end{prop}

It is not known if the weak nonvanishing holds for generalized pairs. Most probably the answer is negative. But there are some cases in which the answer is positive.

\begin{thm}[=Theorem \ref{weaknonvanishing}]
    Let $(X,B+M)$ be a generalized lc generalized pair with data $X' \to X$ and $M'$, and $f:(X,B+M) \rightarrow Z$ be a surjective morphism to a normal variety $Z$.  Assume that 

    $\bullet$ $Z$ has maximal Albanese dimension,

    $\bullet$ $K_X+B+M$ is abundant$/Z$, 

    $\bullet$ $\mathcal{R}(X/Z,K_X+B+M)$ is a finitely generated $\mathcal{O}_Z$-algebra, and

     $\bullet$  $M'$ is semi-ample$/Z$. 

Then, there exists an effective divisor $D \ge 0$ such that $D \equiv K_X+B+M$.
\end{thm}

Since the relative abundance assumption is automatically satisfied when a fibration is relatively of Fano type, we immediately obtain the corollary below.

\begin{cor}[=Corollary \ref{nonvanishing-cor-2}]
	Let $(X,B+M)$ be a generalized lc generalized pair with data $X' \to X$ and $M'$, and $f:(X,B+M) \rightarrow Z$ be a surjective morphism to a normal variety $Z$. Assume that 

$\bullet$ $Z$ has maximal Albanese dimension,

$\bullet$ $\kappa(K_F+B_F+M_F) \ge 0$ where $F$ is a general fiber of $f$, and

$\bullet$ $X$ is of log Fano type over $Z$, that is, there is a boundary $\Delta$ such that $(X,\Delta)$ is lc and $-(K_X+\Delta)$ is ample$/Z$.

Then, there exists an effective divisor $D \ge 0$ such that $D \equiv K_X+B+M$.
\end{cor}

\textbf{Irregular varieties with Albanese fiber of general type.}
From an easy observation one finds that the same argument still works if we slightly weaken the assumption that the base variety $Z$ has maximal Albanese dimension.

We say that $Z$ is irregular \emph{with Albanese fiber of general type} if a general fiber of $a:Z \rightarrow A=\mathrm{Alb}(Z)$ is of general type. 

\begin{prop}
	Let $f:(X,B) \rightarrow Z$ be a surjective morphism from a projective lc pair $(X,B)$ to a normal variety $Z$. Assume that $Z$ is irregular with Albanese fibre of general type and that $K_X+B$ is relatively abundant over $Z$. Then, $K_X+B$ is abundant. Moreover, the subadditvity of Kodaira dimensions $\kappa(K_X+B) \geq \kappa(K_F+B_F)+ \kappa(Z)$ holds, where $F$ is a general fiber, $K_F+B_F= (K_X+B)|_F$, and $\kappa(Z)$ means the Kodaira dimension of a smooth model of $Z$. 
\end{prop}

In the same way Proposition \ref{m-prop} can be generalized as below.

\begin{prop}
	Let $f:(X,B) \rightarrow Z$ be a surjective morphism from a projective lc pair $(X,B)$ to a normal variety $Z$. Assume that $Z$ is irregular with Albanese fibre of general type,  and that $(F,B_F)$ has a good minimal model where $F$ is the general fiber and $K_F+B_F=(K_X+B)|_F$. Then, $K_X+B$ is abundant. Moreover, if the canonical ring $R(K_X+B)$ is finitely generated, for example, if $(X,B)$ is klt, or $\kappa(K_X+B)=0$, then $(X,B)$ has a good minimal model. 
\end{prop}

\textbf{A remark on log Iitaka conjecture.}
As we mentioned before Cao and P\v{a}un [\ref{CP}], Hacon, Popa and Schnell [\ref{HPS}] proved log Iitaka conjecture when $(X,B)$ is klt and $Z=A$ is an abelian variety. By the subadditivity of logarithmic Kodaira dimensions for algebraic fibrations over varieties of general type one can easily deduce log Iitaka conjecture when $(X,B)$ is klt and $Z$ has maximal Albanese dimension (see Remark \ref{remark}). We therefore consider the following conjecture for lc pairs. 

\begin{conj}\label{conj}
Let $f:(X,B) \rightarrow Z$ be a surjective morphism from a projective lc pair $(X,B)$ to a normal variety $Z$. Assume that $Z$ has maximal Albanese dimension. Then, the subadditvity of Kodaira dimensions $\kappa(K_X+B) \geq \kappa(K_F+B_F)+ \kappa(Z)$ holds, where $F$ is a general fiber, $K_F+B_F= (K_X+B)|_F$, and $\kappa(Z)$ means the Kodaira dimension of a smooth model of $Z$. 	
\end{conj}

Extending results from klt pairs to lc pairs is usually much harder than it sounds. One may ask if we can deduce Conjecture \ref{conj} from [\ref{CP}] [\ref{HPS}], when we put an extra positivity assumption on the boundary $B$. The proposition below indicates that adding this extra assumption will not decrease the difficulty of Conjecture \ref{conj}.   

\begin{prop}[See Proposition \ref{relation-conj}]
	Conjecture \ref{conj} holds for lc pairs $(X,B)$ in dimension $\leq n$ if and only if Conjecture \ref{conj} holds for $\mathbb{Q}$-factorial dlt pairs $(Y,B_Y)$ in dimension $\leq n+1$ where $B_Y$ is big$/Z$.
\end{prop}
\vspace{0.3cm}
\subsection*{Ackowledgement}

This work initiated when I visited NCTS, National Taiwan University during the fall of 2014. I would like to thank Professor Jungkai Chen, Professor Caucher Birkar and Dr. Ching-Jui Lai for many helpful discussions and comments. I would also like to thank everyone at NCTS for their hospitality. The main part of this work was done when I was a visiting scholar at CMS, Zhejiang University during the winter of 2014. I am grateful to Professor Kefeng Liu and Professor Hongwei Xu for their advice and support. I am indebted to Professor Osamu Fujino, Dr. Kenta Hashizume and the referees for their useful comments.

\vspace{0.3cm}
\section{Preliminaries}\label{preliminaries}

We work over the complex numbers $\mathbb{C}$. All varieties are quasi-projective and a divisor means a $\Q$-divisor unless stated otherwise. \\

\textbf{Pairs.}
A \emph{pair} $(X/Z,B)$ consists of normal quasi-projective varieties $X$, $Z$, an $\Q$-divisor $B$ on $X$ with
coefficients in $[0,1]$ such that $K_X+B$ is $\mathbb{Q}$-Cartier and a projective morphism $X\rightarrow Z$. If $Z$ is a point or $Z$ is unambiguous in the context, then we simply denote a pair by $(X,B)$.
For a prime divisor $D$ on some birational model of $X$ with a
nonempty centre on $X$, $a(D,X,B)$
denotes the log discrepancy. For definitions and standard results on singularities of pairs
we refer to [\ref{Kollar-Mori}].\\

\textbf{Log minimal models and Mori fibre spaces.}
A pair $(Y/Z,B_Y)$ is a \emph{log birational model} of a pair $(X/Z,B)$ if we are given a birational map
$\phi\colon X\bir Y$ and $B_Y=B^\sim+E$ where $B^\sim$ is the birational transform of $B$ and
$E$ is the reduced exceptional divisor of $\phi^{-1}$, that is, $E=\sum E_j$ where $E_j$ are the
exceptional/$X$ prime divisors on $Y$.
A log birational model $(Y,B_Y)$ is a  \emph{weak lc model} of $(X,B)$ if

$\bullet$ $K_Y+B_Y$ is nef$/Z$, and

$\bullet$ for any prime divisor $D$ on $X$ which is exceptional/$Y$, we have
$$
a(D,X,B)\le a(D,Y,B_Y),
$$

A weak lc model $(Y/Z,B_Y)$ is a \emph{log minimal model} of $(X/Z,B)$ if

$\bullet$ $(Y/Z,B_Y)$ is $\Q$-factorial dlt,

$\bullet$ the above inequality on log discrepancies is strict.

A log minimal model $(Y/Z, B_Y)$ is \emph{good} if $K_Y + B_Y$ is semi-ample$/Z$.\\

On the other hand, a log birational model $(Y/Z,B_Y)$  is called a \emph{Mori fibre space} of $(X/Z,B)$ if

$\bullet$ $(Y/Z,B_Y)$ is $\Q$-factorial dlt,

$\bullet$ there is a $K_Y+B_Y$-negative extremal contraction $Y\to T$
with $\dim Y>\dim T$, and

$\bullet$ for any prime divisor $D$ (on birational models of $X$) we have
$$
a(D,X,B)\le a(D,Y,B_Y)
$$
 and strict inequality holds if $D$ is
on $X$ and contracted$/Y$.\\

Note that our definitions of log minimal models and Mori fibre spaces are slightly different
from the traditional definitions in that we allow $\phi^{-1}$ to contract certain divisors.\\

\textbf{Lc places and lc centers.} 
A prime divisor $T$ over $X$ is said to be \emph{an lc place} of a log pair $(X,B)$ if the log discrepancy $a(T,X,B)=0$. A closed subset $Y$ of $X$ is said to be \emph{an lc center} of $(X,B)$ if there is an lc place $T$ such that the center of $T$ is $Y$.\\

\textbf{Log smooth models.}
A pair $(Y/Z,B_Y)$ is a \emph{log smooth model} of a pair $(X/Z,B)$ if there exists a birational morphism $\pi: Y \rightarrow X$ such that 

$\bullet$ $(Y/Z,B_Y)$ is log smooth,

$\bullet$ $\pi_\ast B_Y=B$,

$\bullet$ $a(E,Y,B_Y)=0$ for every exceptional$/X$ prime divisor on $Y$.

It is obvious that if $(X,B)$ is lc, then $(Y,B_Y)$ is dlt. In this case, it is easy to calculate that $a(D,X,B)\ge a(D,Y,B_Y)$ for any prime divisor $D$ over $X$. Moreover, a log minimal model of $(Y/Z,B_Y)$ is also a log minimal model of $(X/Z,B)$.   \\

\textbf{Ample models and log canonical models.}
Let $D$ be a divisor on a normal variety $X$ over $Z$. A normal variety $T$ is the \emph{ample model}$/Z$ of $D$ if we are given a rational map $\phi: X \bir T$ such that there exists a resolution $X \xleftarrow{p} X' \xrightarrow{q} T$ with 

$\bullet$ $q$ being a contraction,

$\bullet$ $p^\ast D \sim_\R q^\ast D_T+E$ where $D_Z$ is an ample$/Z$ divisor on $T$, and

$\bullet$ for every divisor $B \in |p^\ast D/Z|_\R$, then $B \geq E$. 

Note that the ample model is unique if it exists. The existence of the ample model is equivalent to that the divisorial ring $R(D)$ is a finitely generated $\mathcal{O}_Z$-algebra when $D$ is $\Q$-Cartier. \\

A normal variety $T$ is the \emph{log canonical model}$/Z$ (lc model for short) $T$ of a pair $(X/Z,B)$ if it is the ample model$/Z$ of $K_X+B$. The existence of the lc model of a klt pair is ensured by [\ref{BCHM}], and the existence of the lc model of an lc pair is still an open question. It is known that proving the existence of the lc models of lc pairs $(X,B)$ with $K_X+B$ being big is equivalent to proving the abundance conjecture for klt pairs.\\

\textbf{Nakayama-Zariski decompositions.}
Nakayama [\ref{Nakayama}] defined a decomposition $D=P_\sigma(D)+N_\sigma(D)$ for any pseudo-effective
$\R$-divisor $D$ on a smooth projective variety. We refer to this as the Nakayama-Zariski decomposition.
We call $P_\sigma$ the positive part and $N_\sigma$ the negative part. We can
extend it to the singular case as follows.
Let $X$ be a normal projective variety and $D$ a pseudo-effective $\R$-Cartier divisor on $X$. We define $P_\sigma(D)$
by taking a resolution $f\colon W\to X$ and letting $P_\sigma(D):=f_*P_\sigma(f^*D)$. A divisor $D$ is \emph{pseudo-movable} if $D= P_\sigma(D)$, i.e. $D \in \overline{\mathrm{Mov}(X)}$. \\

\textbf{Generalized polarized pairs.}
The notion of polarized pairs was introduced in Birkar-Hu [\ref{BH-I}], and then generalized in Birkar-Zhang [\ref{BZ}]. A \emph{generalized (polarized) pair} $(X/Z,B+M)$ with data $X' \to X$ and $M'$ consists of a normal variety $X$, a projective morphism $X \rightarrow Z$, a boundary divisor $B$ and a Weil divisor $M$ together with a birational morphism $\pi: (X', B'+M') \rightarrow X$ such that 

$\bullet$ $K_{X'}+B'+M'=\pi^\ast (K_X+B+M)$,

$\bullet$ $(X',B')$ is sub-dlt and $B=\pi_\ast B'\geq 0$,

$\bullet$ $M'$ is nef$/Z$ and $M= \pi_\ast M'$.\\

A generalized pair $(X/Z,B+M)$ is said to \emph{have the abundant moduli part} if $M'$ is nef and abundant. Note that if $\pi$ can be chosen to be the identity morphism, then we call $(X/Z,B+M)$ \emph{a polarized pair}. \\

\textbf{Generalized singularities.}
We refer to [\ref{BZ}] for the details of generalized singularities. Let $(X,B+M)$ be a generalized pair and we use the notation $(X',B'+M')$ as above. Given a prime divisor $D$ over $X$, the \emph{generalized log discrepancy} $a(D,X,B+M)$ is defined to be the "usual" log discrepancy $a(D,X',B')$. Note that $B'$ may contain components with negative coefficients. We say that $(X,B+M)$ is \emph{generalized klt} (resp. generalized lc) if $a(D,X,B+M)>0$ (resp. $a(D,X,B+M) \ge 0$) for every prime divisor $D$ over $X$. This is also equivalent to saying that $(X',B')$ is sub-klt (resp. sub-lc). Moreover, one can naturally define a \emph{generalized dlt} generalized pair. Note that the existence of $\Q$-factorial dlt blow-up of a generalized lc generalized pair is ensured by [\ref{B-lc-flips}, Theorem 3.5].
\\

\textbf{Minimal models of generalized pairs.}
A generalized pair $(Y,B_Y+M_Y)$ together with a birational map $\phi: X \dashrightarrow Y$ is a \emph{log minimal model} of $(X,B+M)$ (where $B_Y=B^\sim+E$ as before) if

$\bullet$ $(Y,B_Y+M_Y)$ is $\Q$-factorial dlt,

$\bullet$ $K_Y+B_Y+M_Y$ is nef,

$\bullet$ for any prime divisor $D$ (on birational models of $X$) we have
$$
a(D,X,B+M)\le a(D,Y,B_Y+M_Y)
$$
and strict inequality holds if $D$ is
on $X$ and contracted$/Y$. 

On the other hand we call $(Y,B_Y+M_Y)$ a Mori fiber space of $(X,B+M)$ if it satisfies the conditions above with the condition $K_Y+B_Y+M_Y$ being nef replaced by 

$\bullet$ there is a $(K_Y+B_Y+M_Y)$-negative extremal contraction $Y\to T$
with $\dim Y>\dim T$. 

\vspace{0.3cm}
\section{Proof of The Main Theorem}\label{proof-1}

We will use a canonical bundle formula for lc pairs. Recall the following result from [\ref{Hacon-Xu}, Theorem 2.1], which is a generalization for lc pairs of the main result of Fujino and Mori [\ref{FM}].

\begin{thm}[$\mathtt{[}$\ref{Hacon-Xu}, Theorem 2.1$\mathtt{]}$]\label{c-bdle-form}
	Let $f:(X,B) \rightarrow Z$ be a projective morphism from an lc pair to a normal variety $Z$, $B$ be a $\mathbb{Q}$-boundary divisor. Then, there exists a commutative diagram of projective morphisms 
	$$
	\xymatrix{
		 X' \ar[rr]^{\pi}\ar[d]_{f'} &   & X \ar[d]^{f} \\
		 Y  \ar[rr]^{\mu} & & Z   &}
	$$
	with the following properties. \\
	(1). $\pi$ is a birational morphism; \\
	(2). there exists a $\mathbb{Q}$-divisor $B'$ on $X'$ with the coefficients $\leq 1$ such that $(X',B')$ is $\mathbb{Q}$-factorial dlt and  \\
	$$
	\pi_\ast \mathcal{O}_{X'}(m(K_{X'}+B'))\cong \mathcal{O}_X(m(K_X+B)), \forall m \in \mathbb{N};
	$$
	(3). there exist a boundary $\mathbb{Q}$-divisor $B_Y$ and a $\mathbb{Q}$-divisor $M_Y$ on $Y$ such that $(Y,B_Y+M_Y)$ is a dlt polarized pair and that $K_Y+B_Y+M_Y$ is big$/Z$; \\
	(4). there exist an effective $\mathbb{Q}$-divisor $R$ on $X'$ such that $f'_\ast \mathcal{O}_{X'}(\lfloor iR\rfloor)\cong \mathcal{O}_{Y}$ for all $i \in \mathbb{N}$ and 
	$$
	K_{X'}+B'\sim_{\mathbb{Q}}f'^\ast(K_{Y}+B_{Y}+M_Y) +R;
	$$
	(5). each component of $\lfloor B_Y \rfloor$ is dominated by a vertical component of $\lfloor B' \rfloor$.
\end{thm}

\begin{rem}[Reduction to klt case]\label{rem-klt}
Let $(X/Z,B)$, $(X'/Z,B')$, $(Y/Z,B_Y+M_Y)$, $f$ and $f'$ be as in the previous theorem. It is obvious that 
$$
\kappa(K_X+B/Z)=\kappa(K_{X'}+B'/Z)=\kappa(K_Y+B_Y+M_Y/Z)= \dim Y -\dim Z.
$$
Let $P'$ be the vertical$/Y$ part of $\lfloor B' \rfloor$. Since $\kappa(K_{X'}+B' -\epsilon P'/Y)=\kappa(K_{X'}+B'/Y)=0$ for every sufficiently small number $\epsilon$, we can apply the canonical bundle formula Theorem \ref{c-bdle-form} (cf. [\ref{FM}, Section 4]) again to obtain a generalized klt generalized pair $(Y,B_Y^\epsilon+M_Y^\epsilon)$ such that 
$$
K_{X'}+B' -\epsilon P' \sim_{\mathbb{Q}}f'^\ast(K_{Y}+B_{Y}^\epsilon+M_Y^\epsilon ) +R^\epsilon
$$ 
where $R^\epsilon$ is an effective $\Q$-divisor with $f'_\ast \mathcal{O}_{X'}(\lfloor iR^\epsilon \rfloor)\cong \mathcal{O}_{Y}$ for all $i \in \mathbb{N}$.
By the construction of [\ref{FM}, 4.4] we see that $(Y,B_Y^\epsilon+M_Y^\epsilon)$ is a klt polarized pair, that is, $(Y,B_Y^\epsilon)$ is klt and $M_Y^\epsilon$ is nef. Moreover, since $K_Y+B_Y+M_Y$ is big$/Z$, $K_Y+B_Y^\epsilon+M_Y^\epsilon$ is also big$/Z$. So, we have 
$$
\kappa(K_X+B/Z)=\kappa(K_{X'}+B'-\epsilon P'/Z)=\kappa(K_Y+B_Y^\epsilon+M_Y^\epsilon/Z)= \dim Y -\dim Z.
$$
\end{rem}

We also need the following canonical bundle formula for lc-trivial fibrations from [\ref{Amb}] or [\ref{Fujino-Gongyo}] as well.  

\begin{thm}$\mathrm{([\ref{Fujino-Gongyo}, ~Theorem~1.1], [\ref{Amb}, ~Theorem~3.3]).}$\label{c-bdle-form-2}
	Let $f:(X,B) \rightarrow Y$ be a projective morphism from an lc pair to a normal variety $Y$ with connected fibers, $B$ be a $\mathbb{Q}$-boundary divisor. Assume that $K_X+B \sim_{\mathbb{Q}} 0/Y$. Then there exists a boundary $\mathbb{Q}$-divisor $B_Y$ and a $\mathbb{Q}$-divisor $M_Y$ on $Y$ satisfying the following properties. \\
	(1). $(Y,B_Y+M_Y)$ is a generalized lc generalized pair with data $Y' \to Y$ and $M_{Y'}$;\\
	(2). $M_Y'$ is nef and abundant; \\
	(3). $K_X+B \sim_\mathbb{Q} f^\ast (K_Y+B_Y+M_Y)$.
\end{thm}

The following lemma is obvious.

\begin{lem}\label{lem1}
Let $D$ be a pseudo-effective divisor and $P \ge 0$ be an effective divisor on a projective normal variety $X$. If $\kappa(D-\epsilon P) \ge 0$ for some small rational number $\epsilon >0$, then $\kappa(D -\epsilon' P)=\kappa(D)$ for every small rational number $\epsilon' <\epsilon$.  If $\kappa_\sigma(D-\epsilon P) \ge 0$ for some small rational number $\epsilon >0$, then $\kappa_\sigma(D -\epsilon' P)=\kappa_\sigma(D)$ for every small rational number $\epsilon' <\epsilon$. 
\end{lem}

\begin{lem}[Nonvanishing]\label{nonvanishing}
	Let $f: (X,B) \rightarrow Z$ be a surjective morphism from a projective lc pair $(X,B)$ to a normal variety $Z$. Assume that $Z$ has maximal Albanese dimension. If $\kappa(K_X+B/Z) \geq 0$, then $\kappa(K_X+B) \geq 0$.
\end{lem}

\begin{proof}
Replacing $(X,B)$ with a dlt blow-up we can assume that $(X,B)$ is $\mathbb{Q}$-factorial dlt. Let $g: X \dashrightarrow Y$ be a relative Iitaka fibration of $K_X+B$ over $Z$. We now apply Theorem \ref{c-bdle-form} to obtain a commutative diagram 
$$
\xymatrix{
	X \ar@{-->}[d]_{g} &   & X' \ar[d]^{g'}\ar[ll]^{\pi} \\
	Y   & & Y' \ar[ll]^{\mu}  &}
$$
a $\mathbb{Q}$-factorial dlt pair $(X',B')$ and a $\mathbb{Q}$-factorial dlt polarized pair $(Y',B_{Y'}+M_{Y'})$ such that $K_{X'}+B'\sim_{\mathbb{Q}}g'^\ast(K_{Y'}+B_{Y'}+M_{Y'}) +R$ and $K_{Y'}+B_{Y'}+M_{Y'}$ is big$/Z$. Let $P'$ be the vertical$/Y'$ part of $\lfloor B'\rfloor$. As we discussed in Remark \ref{rem-klt}, if we apply Theorem \ref{c-bdle-form} to $(X',B' -\epsilon P')$, then we obtain a klt polarized pair $(Y',B_{Y'}^\epsilon +M_{Y'}^\epsilon )$ such that $K_{Y'}+B_{Y'}^\epsilon +M_{Y'}^\epsilon  $ is big$/Z$ when $\epsilon>0$ is sufficiently small. By [\ref{BC}, Theorem 4.1] and [\ref{CKP}] we have $\kappa(K_{Y'}+B_{Y'}+M_{Y'}) \geq \kappa(K_{Y'}+B_{Y'}^\epsilon +M_{Y'}^\epsilon) \geq 0$ which in turn implies that $\kappa(K_X+B) \geq 0$.
\end{proof}

\begin{thm}\label{m-thm'}
Let $f:(X,B) \rightarrow Z$ be a surjective morphism from a projective lc pair $(X,B)$ to a normal variety $Z$. Assume that $Z$ has maximal Albanese dimension and that $K_X+B$ is relatively abundant over $Z$. Then, $K_X+B$ is abundant. Moreover, the subadditvity of Kodaira dimensions $\kappa(K_X+B) \geq \kappa(K_F+B_F)+ \kappa(Z)$ holds, where $F$ is a general fiber, $K_F+B_F= (K_X+B)|_F$, and $\kappa(Z)$ means the Kodaira dimension of a smooth model of $Z$. 
\end{thm}

\begin{proof}
Replacing $(X,B)$ with a dlt blow-up we can assume that $(X,B)$ is $\mathbb{Q}$-factorial dlt. Let $g: X \dashrightarrow Y$ be a relative Iitaka fibration of $K_X+B$ over $Z$. By replacing $(X,B)$ with a suitable dlt blow-up again we can assume that $g$ is a morphism. Since $K_X+B$ is relatively abundant over $Z$, we have that $\kappa_\sigma (K_X+B/Y) =\kappa(K_X+B/Y)=0$ by [\ref{Lehmann}, Theorem 6.1] (note that the theorem we referred to also applies to singular varieties as we can lift divisors to smooth models) and hence $\kappa_\sigma (K_G+B_G) =\kappa(K_G+B_G)=0$ where $G$ is the generic fiber of $g$. In particular, $K_G+B_G \equiv N_\sigma(K_G+B_G)$. Now run an LMMP$/Y$ on $K_X+B$ with scaling of some ample divisor. Thanks to [\ref{BH-I}], [\ref{Gongyo}] and [\ref{Kawamata}] this process will reach a model $(X',B')$ such that $K_{G'}+B'_{G'} \sim_\mathbb{Q} 0$ after finite steps where $G'$ is the generic fiber of $g': X' \rightarrow Y$. By replacing $X$, $B$ and $G$ with $X'$, $B'$ and $G'$, we can assume that there exists a nonempty open subset $U$ of $Y$ such that $K_X+B \sim_{\mathbb{Q},U} 0$.  

Let $P$ be the vertical$/Y$ part of $\lfloor B \rfloor$, and let $\epsilon>0$ be a sufficiently small number. We claim that $\kappa(K_{X}+B)=\kappa(K_{X}+B-\epsilon P)$ and $\kappa_\sigma(K_{X}+B)=\kappa_\sigma(K_{X}+B-\epsilon P)$. Thanks to Theorem \ref{c-bdle-form} there exist a birational model $\pi: X'' \rightarrow X$ and a birational model $\mu: Y' \rightarrow Y$ together with a morphism $g'':X'' \rightarrow Y'$
$$
\xymatrix{
	X'' \ar[rr]^{\pi}\ar[d]_{g''} &   & X \ar[d]^{g} \\
	Y'  \ar[rr]^{\mu} & & Y   &}
$$
such that 
$$
\pi^\ast (K_X+B) + E \sim_\mathbb{Q} g''^\ast (K_{Y'}+B_{Y'}+M_{Y'})+R
$$ 
where $E$ is exceptional$/X$ and $K_{Y'}+B_{Y'}+M_{Y'}$ is big$/Z$. Because $P$ is vertical$/Y$, $\pi^\ast P$ is vertical$/Y'$. Pick a $\mathbb{Q}$-divisor $D_{Y'}$ on $Y'$ such that $g''^\ast D_{Y'} \geq \pi^\ast P$. Note that $K_{Y'}+B_{Y'} -\epsilon D_{Y'} +M_{Y'}$ is big$/Z$ since $\epsilon$ is sufficiently small, and hence $\kappa(K_X+B-\epsilon P/Z)=\kappa(K_X+B/Z) \ge 0$. By Lemma  \ref{nonvanishing} one deduces that $\kappa(K_X+B-\epsilon P)\ge 0$, and by Lemma \ref{lem1} one deduces the claim.

Now we run an LMMP$/Y$ on $K_X+B- \epsilon P$ which terminates at a good minimal model 
$$
\xymatrix{
	X' \ar[rr]^{h}\ar[dr]_{g'} &   & T \ar[dl]^{\mu} \\
	& Y &} 
$$
by [\ref{B-lc-flips}, Theorem 1.4]. Because $K_X+B \sim_{\mathbb{Q},U} 0$, it is clear that $\kappa(K_{X'}+B')=\kappa(K_{X'}+B'-\epsilon P')$ and $\kappa_\sigma(K_{X'}+B')=\kappa_\sigma(K_{X'}+B'-\epsilon P')$ where $B'$ and $P'$ are the birational transforms of $B$ and $P$ respectively. Moreover, $\mu: T \rightarrow Y$ is birational. 
It suffices to show that $K_{X'}+B'-\epsilon P'$ is abundant, namely $\kappa (K_{X'}+B'-\epsilon P')=\kappa_\sigma(K_{X'}+B'-\epsilon P')$. To this end we apply Theorem \ref{c-bdle-form-2} on $(X',B'-\epsilon P')$ and get a generalized klt generalized pair $(T,\Delta_{T}+M_T)$ with the abundant moduli part such that 
$$
K_{X'}+B'-\epsilon P' \sim_\mathbb{Q} h^\ast(K_T+B_T+M_T). 
$$
Because $K_T+\Delta_T+M_T$ is big$/Z$ and there exists a boundary $B_T$ such that $(T,B_T)$ is klt and $K_T +B_T \sim_\Q K_T+\Delta_T+M_T$, we can apply [\ref{BC}, Theorem 1.1] to achieve the conclusion. In particular, we obtain
\begin{align*}
\kappa(K_X+B) &= \kappa (K_{X'}+B' -\epsilon P') \\
 & =\kappa(K_T+B_T) \\
 &\geq \dim T- \dim Z+ \kappa(Z) \\
  &= \kappa (K_F+B_F) + \kappa(Z).
\end{align*} 
\end{proof}

\begin{rem}
We point out that a crucial step in the argument in [\ref{BC}] relies on the extension theorem from [\ref{DHP}] which was obtained via an analytic method. So, both Theorem \ref{m-thm} and [\ref{BC}] do not have a pure algebraic proof at this point. One wonders if these theorems can be argued in a parallel way as in [\ref{CH}].
\end{rem}
Next we prove some corollaries.

\begin{cor}\label{m-cor'}
	Let $f:(X,B) \rightarrow A$ be a morphism from a projective lc pair $(X,B)$ to an abelian variety $A$. Assume that $(X,B)$ relatively has a good minimal model over $A$. Then, $(X,B)$ has a good minimal model. In particular, if $K_X+B$ is semi-ample$/A$, then it is semi-ample.
\end{cor}

\begin{proof}
	Replacing $(X,B)$ with a good minimal model$/A$ one can assume that $K_X+B$ is semi-ample$/A$ and globally nef (see [\ref{BC}, \S 3]).  Let $X \rightarrow Z \rightarrow A$ be the Stein factorization. It is easy to see that $Z$ has maximal Albanese dimension and $K_X +B$ is semi-ample$/Z$. Therefore $K_X +B$ is globally abundant by Theorem \ref{m-thm}. We can apply a similar argument for every lc center of $(X,B)$ and then obtain that $K_X +B$ is nef and log abundant. Finally we establish the result by [\ref{Fujino-Gongyo}, Theorem 1.5] and [\ref{Hacon-Xu-2}]. 
\end{proof}

We also give a generalization of [\ref{Lai}, Theorem 4.2 and Corollary 4.3] as Corollary \ref{s-cor} and Corollary \ref{s-cor-2}. We present them below for the reader's convenience.

\begin{cor}
	Let $a:(X,B) \rightarrow A$ be the Albanese map from a projective klt pair $(X,B)$ to its Albanese variety $A=\mathrm{Alb}(X)$. Suppose that $(F,B_F)$ has a good minimal model where $F$ is the general fiber and $K_F+B_F=(K_X+B)|_F$. Then, $(X,B)$ has a good minimal model.
\end{cor}

\begin{cor}
	Let $f:(X,B) \rightarrow Z$ be a surjective morphism from a projective klt pair $(X,B)$ to a normal variety $Z$. Suppose that $Z$ has maximal Albanese dimension, and that $(F,B_F)$ has a good minimal model where $F$ is the general fiber and $K_F+B_F=(K_X+B)|_F$. Then, $(X,B)$ has a good minimal model.
\end{cor}

In fact one can prove a stronger version for lc pairs instead of klt pairs.

\begin{prop}\label{m-prop'}
	Let $f:(X,B) \rightarrow Z$ be a surjective morphism from a projective lc pair $(X,B)$ to a normal variety $Z$. Suppose that $Z$ has maximal Albanese dimension, and that $(F,B_F)$ has a good minimal model where $F$ is the general fiber and $K_F+B_F=(K_X+B)|_F$. Then, $K_X+B$ is abundant. Moreover, if the canonical ring $\mathcal{R}(K_X+B)$ is finitely generated, for example, if $(X,B)$ is klt, or $\kappa(K_X+B)=0$, then $(X,B)$ has a good minimal model. 
\end{prop}

\begin{proof}
 Because $(F,B_F)$ has a good minimal model, we have that $K_F+B_F$ is abundant. By semi-continuity theorem we deduce that $K_G+B_G$ is abundant where $G$ is the generic fibre of $f$. We therefore conclude that $K_X+B$ is abundant thanks to Theorem \ref{m-thm}.
	
	In particular, if the canonical ring $R(K_X+B)$ is finitely generated, then $K_X+B$ birationally has a Nakayama-Zariski decomposition with semi-ample positive part by [\ref{Lehmann}, Proposition 6.4] (note that the proposition we referred to also applies to singular varieties as we can lift line bundles to smooth models). Therefore by [\ref{BH-I}, Theorem 1.1] we deduce $(X,B)$ has a good minimal model.
\end{proof}

\vspace{0.3cm}
\section{Variants}\label{variants}
\subsection{Variant: generalized pairs}
It is known [\ref{BH-I}], [\ref{BZ}] that minimal model theory works for generalized lc generalized pairs. In [\ref{BH-I}] we discussed that the abundance and the nonvanishing are not expected in general. However, there are some interesting cases that abundance is expected (for example, $(X,B+M)$ has the abundant moduli part).

\begin{lem}[Weak Nonvanishing]\label{nonvanishing2}
	Let $f: (X,B+M) \rightarrow Z$ be a surjective morphism from a projective generalized lc generalized pair $(X,B+M)$ to a normal variety $Z$, $B$ be a $\mathbb{Q}$-boundary divisor. Assume that $Z$ has maximal Albanese dimension and that $K_X+B+M$ is big$/Z$. Then, there exists an effective divisor $D \ge 0$ such that $D \equiv K_X+B+M$.
\end{lem}

\begin{proof}
Replacing $(X,B+M)$ with some log resolution we can assume that $(X,B+M)$ is a dlt polarized pair. Let $P:= \lfloor B \rfloor$. Pick a sufficiently small number $\epsilon$ and consider the polarized pair $(X+B-\epsilon P+M)$. By [\ref{BC}, Theorem 4.1] there exists an effective divisor $D' \ge 0$ such that $K_X+B-\epsilon P+M \equiv D'$. We simply take $D= D'+\epsilon P$ to conclude this lemma.
\end{proof}

One may ask if the subadditivity of Kodaira dimensions Conjecture \ref{conj} holds for a generalized polairzed pair $(X,B+M)$. The answer is positive if we suppose a strong positivity assumption on both $K_X+B+M$ and $M$.

\begin{prop}\label{prop1}
	Let $(X,B+M)$ be a generalized lc generalized pair with data $X' \to X$ and $M'$, and $f:(X,B+M) \rightarrow Z$ be a surjective morphism to a normal variety $Z$. Assume that 
	
	$\bullet$ $Z$ has maximal Albanese dimension, 
	
    $\bullet$ $K_X+B+M$ is big$/Z$, and 
    
    $\bullet$ $M'$ is abundant.
	
	Then, $K_X+B+M$ is abundant. Moreover, the subadditvity of Kodaira dimensions $\kappa(K_X+B+M) \geq \kappa(K_F+B_F+M_F)+ \kappa(Z)$ holds, where $F$ is a general fiber, $K_F+B_F+M_F= (K_X+B+M)|_F$, and $\kappa(Z)$ means the Kodaira dimension of a smooth model of $Z$. 
\end{prop}

\begin{proof}
Replacing $(X,B+M)$ with a birational model we can assume that $(X,B+M)$ is a log smooth dlt polarized pair. Moreover, we can assume that $M$ is nef and abundant. Let $P:= \lfloor B \rfloor$, $\epsilon>0$ be a sufficiently small rational number and let $\Delta_\epsilon$ be a $\mathbb{Q}$-boundary divisor such that $(X,\Delta_\epsilon)$ is klt and 
$$
K_X+\Delta_\epsilon \sim_\mathbb{Q} K_X+B-\epsilon P +M.
$$
Since $\epsilon$ can be chosen arbitrary small, $K_X+\Delta_\epsilon$ is big$/Z$. Due to Lemma \ref{nonvanishing} and Lemma \ref{lem1} we get that $\kappa(K_X+\Delta_\epsilon)= \kappa(K_X+B+M)$ and $\kappa_\sigma (K_X+\Delta_\epsilon)= \kappa_\sigma (K_X+B+M)$. It suffices to show that $K_X+\Delta_\epsilon$ is abundant which follows from [\ref{BC}, Theorem 1.1].  
\end{proof}

\begin{rem}\label{rem-gpp}
It is natural to ask if we can remove one of these assumptions above. The example below shows that Proposition \ref{prop1} fails if we drop the condition $M$ being abundant. 
\end{rem}

\begin{exa}\label{exa1}
Let $E$ be an elliptic curve and $\mathcal{E}$ be a nontrivial unipotent vector bundle over $E$ of rank two with $\mathcal{E}$ normalized and $\deg \wedge(\mathcal{E})=0$. Let $\pi: S:= \mathbb{P}_E(\mathcal{E}) \rightarrow E$ be a ruled surface over the base curve $E$, and $E'$ be the divisor corresponding to the nontrivial global section of $\mathcal{E}$. Then, we have the following properties of $E'$: 

(1). $E'^2=0$, so $E'$ is nef;

(2). $K_S+2E' \sim_\mathbb{Q} 0$;

(3). $\kappa(E')=0$.

Note that $K_S+3E'$ is $\pi$-big and nef. However, $\kappa(K_S+3E')=0 <1= \kappa(K_F+3E'|_F)$ where $F$ is a general fiber.  
\end{exa}

\begin{rem}\label{rem-gpp-2}
	By modifying Example \ref{exa1} as below we note that Proposition \ref{prop1} also fails when $K_X+B+M$ is not big$/Z$.
\end{rem}

\begin{exa}\label{exa2}
Let $\pi: S \rightarrow E$ be the $\mathbb{P}^1$-fibration over an elliptic curve $E$ as in Example \ref{exa1}.  Consider a $\mathbb{P}^1$-bundle over $S$
$$
q: Y=\mathbb{P}(\mathcal{O}_S \oplus \mathcal{O}_S (-1)) \rightarrow S.
$$
Let $p: Y \rightarrow C=C(S)$ be the birational contraction of the negative section
$S'$ on $Y$ and $H$ a general sufficiently ample $\mathbb{Q}$-divisor on $C$ such that $\lfloor H \rfloor = 0$ and $K_Y +S'+3 E_Y'+H'$ is big where $E_Y'$ is the birational tranform of the cone of $E'$ on $C$ and $H'=p^\ast H$. Now set 
$$
\lambda = \inf\{t| \text{ $K_Y +S'+3E_Y'+ t H'$ is pseudo-effective$/S$} \}
$$
and $\Delta_Y=S'$ and $M_Y=3E_Y'+ \lambda H'$. We have that $M_Y$ is nef and big, and $K_Y+\Delta_Y+ M_Y =\pi^\ast (K_S+ 3E')$. Therefore $\kappa(K_Y+\Delta_Y+M_Y)=0 < 1= \kappa(K_F+\Delta_F +M_F) +\kappa(E)$ where $F$ is a general fiber.
\end{exa}

In the example above we note that $K_Y+B_Y+M_Y$ is relatively abundant over $E$ and $M_Y$ is big. If $(Y,B_Y) $ is klt, then there is a boundary $\Delta_Y$ such that $(Y,\Delta_Y)$ is klt and $K_Y+\Delta_Y \sim_\Q K_Y+B_Y+M_Y$. Hence the subadditivity of Kodaira dimensions in turn follows from Theorem \ref{m-thm}. From this discussion we see that the existence of an lc center (more precisely, a horizontal lc center) fails the subadditivity of Kodaira dimensions.  

\begin{prop}\label{prop2}
	Let $(X,B+M)$ be a generalized lc generalized pair with data $X' \to X$ and $M'$, and $f:(X,B+M) \rightarrow Z$ be a surjective morphism to a normal variety $Z$. Assume that 
	
	$\bullet$ $Z$ has maximal Albanese dimension,
	
	$\bullet$ $K_X+B+M$ is abundant$/Z$,
	
	$\bullet$ $M'$ is abundant, and
	
	$\bullet$ every lc center of $(X,B+M)$ is vertical$/Z$. 
	
	Then, $K_X+B+M$ is abundant. Moreover, the subadditvity of Kodaira dimensions $\kappa(K_X+B+M) \geq \kappa(K_F+B_F+M_F)+ \kappa(Z)$ holds, where $F$ is a general fiber, $K_F+B_F+M_F= (K_X+B+M)|_F$, and $\kappa(Z)$ means the Kodaira dimension of a smooth model of $Z$. 
\end{prop}

\begin{proof}
The argument almost follows the lines of the proof of Theorem \ref{m-thm}. Replacing $(X,B+M)$, we can assume that it is a $\Q$-factorial dlt polarized pair. Moreover, there is a morphism $g: X \rightarrow Y$ over $Z$ and a nonempty open subset $U$ of $Y$ such that $K_X+B\sim_{\mathbb{Q},U} 0$. Let $P=\lfloor B \rfloor$, and pick a sufficiently small number $\epsilon >0$. There is a boundary $\Delta$ such that $(X,\Delta)$ is klt and $K_X+\Delta \sim_{\mathbb{Q}} K_X+B -\epsilon P +M$. It is easy to see $\kappa(K_X+\Delta)=\kappa(K_X+B+M)$ and $\kappa_\sigma(K_X+\Delta)=\kappa_\sigma(K_X+B+M)$. Now run an LMMP$/Y$ on $K_X+\Delta$ which terminates at a good minimal model $h:X' \rightarrow T$. It follows that $K_{X'}+\Delta' \sim_{\mathbb{Q}} h^\ast(K_T+\Delta_T)$ for some klt pair $(T,\Delta_T)$. We see that $K_T+\Delta_T$ is big$/Z$ which in turn implies the abundance of $K_{X'}+\Delta'$ and hence the abundance of $K_X+B+M$.
\end{proof}

We point out that the abundance of the moduli part portrays a central role in the proof of Proposition \ref{prop1} and \ref{prop2}. Especially Example \ref{exa1} shows that the subadditivity of Kodaira dimensions fails if we drop this assumption. On the other hand, Lemma \ref{nonvanishing2} indicates that the weak nonvanishing still holds if we suppose a certain positivity assumption on $K_X+B+M$ relatively over $Z$. \\

\textbf{GV-sheaves and Fourier-Mukai tranforms.}
We briefly describe some basic concepts from the theory of abelian varieties. For any smooth projective variety $X$, we will denote by $a: X \rightarrow A$
the Albanese morphism of $X$ and $\widehat{A} = \mathrm{Pic}^0
(X)$ the dual of the Albanese
variety. We will denote by $\mathbf{D}(X)$ the bounded derived category of coherent
sheaves on $X$. For an abelian variety $A$ and its dual $\widehat{A}$, we denote by $\mathscr{P}_A$ the normalized
Poincar\'e line bundle on $A \times \widehat{A}$. For $\alpha \in \widehat{A}$ we denote by $P_\alpha$ the line bundle
that represents $\alpha$. By [\ref{Mukai}], the following functors give equivalence between
$\mathbf{D}(A)$ and $\mathbf{D}(\widehat{A})$:
$$
\mathbf{R}\Phi_{\mathscr{P}_A}: \mathbf{D}(A) \rightarrow \mathbf{D}(\widehat{A}), \mathbf{R}\Phi_{\mathscr{P}_A}(\cdot)=\mathbf{R} p_{\widehat{A},\ast}(p_A^\ast(\cdot) \otimes \mathscr{P}_A).
$$
$$
\mathbf{R}\Psi_{\mathscr{P}_A}: \mathbf{D}(A) \rightarrow \mathbf{D}(\widehat{A}), \mathbf{R}\Psi_{\mathscr{P}_A}(\cdot)=\mathbf{R} p_{A,\ast}(p_{\widehat{A}}^\ast(\cdot) \otimes \mathscr{P}_A).
$$

For any coherent sheaf $\mathscr{F}$ on $X$ and any morphism $f : X \rightarrow A$ to an
abelian variety, we define the \emph{$i$-th cohomological locus}
$$
V^i(\mathscr{F}, f) := \{ \alpha \in \widehat{A} | H^i(X, \mathscr{F} \otimes P_\alpha)  \neq 0 \}.
$$
If $f = a$ is the Albanese morphism, we will simply denote by $V^i(\mathscr{F})$ the
$i$-th cohomological locus. For more definitions and results we refer to [\ref{PP}] and [\ref{CJ}]. 

For any ample line bundle $L$ on $\widehat{A}$, the isogeny $\phi_L : \widehat{A} \rightarrow A$ is defined by
$\phi_L (\hat{a}) = t^\ast_{\hat{a}} L^\vee \otimes L$. Let $L$ be the vector bundle on $A$ defined by
$$
\widehat{L}^\vee :=  p_{\widehat{A},\ast}(p_A^\ast L \otimes \mathscr{P}_A).
$$
One has that
$$
\phi_L^\ast (\widehat{L}^\vee) \cong \bigoplus\limits_{h^0(L)} L.
$$

Let $\mathscr{F}$ be a coherent sheaf on an abelian variety $A$. Then, $\mathscr{F}$ is a \emph{GV-sheaf} if
$\codim_{\widehat{A}} \Supp \mathbf{R}^i \Phi_{\mathscr{P}_A}
(\mathscr{F}) \ge i$ for all $i \ge 0$. The main result of [\ref{Hacon}, Theorem 1.2] (cf. [\ref{PP}, Theorem A]) asserts that, if for any sufficiently ample line bundle $L$ on $\widehat{A}$, 
$$
H^i(A, \mathscr{F} \otimes \widehat{L}^\vee) = 0,  i > 0  ,
$$
then $\mathscr{F}$ is a GV-sheaf. In particular, one has inclusions
$$
V^0(\mathscr{F}) \supset V^1(\mathscr{F}) \supset \ldots \supset V^n(\mathscr{F}).
$$\\

\textbf{A theorem on weak nonvanishing.} As we mentioned previously some weak positivity assumption on both $K_X+B+M$ and $M$ relatively over $Z$ implies the weak nonvanishing. We first prove an easy lemma.

\begin{lem}[cf. \text{[\ref{BH-I}, Theorem 1.1]}]\label{lem-NZ}
Let $(X,B+M)$ be a generalized lc generalized pair. Suppose that $K_X+B+M$ birationally has a Nakayama-Zariski decomposition with nef positive part. Then $(X,B+M)$ has a log minimal model.
\end{lem}

\begin{proof}
Replacing $(X,B+M)$ we can assume that $P=P_\sigma(K_X+B+M)$ is nef. Now run an LMMP on $K_X+B+M +\alpha P$ with scaling of some ample divisor for some $\alpha \gg 0$. This LMMP is $P$-trivial due to [\ref{BH-I}, Theorem 3.2, Proof of Theorem 1.1]. By an easy computation we deduce that it terminates with a log minimal model. 
\end{proof}

\begin{lem}\label{lc-m-g-model}
	Let $(X/Z,B+M)$ be a generalized lc generalized pair. Assume that the divisorial ring $\mathcal{R}(X/Z,K_X+B+M)$ is a finitely generated $\mathcal{O}_Z$-algebra, and that $K_X+B+M$ is abundant$/Z$. Then, $(X/Z,B+M)$ has a log minimal model $(X'/Z,B'+M')$ on which $K_{X'}+B'+M'$ is semi-ample over $Z$.
\end{lem}

\begin{proof}
	We first treat the case when $Z$ is a point. Since $\mathcal{R}(K_X+B+M)$ is finitely generated, there is a log resolution $f:Y \rightarrow X$ such that $f^\ast (K_X+B+M)=P+N$ where $P$ is semi-ample and $N$ is the asymptotic fixed part. Note that the abundance of $K_X+B+M$ implies that $N= N_\sigma (f^\ast(K_X+B+M))$ by [\ref{Lehmann}, Proposition 6.4]. Therefore, $K_X+B+M$ birationally has a Nakayama-Zariski decomposition with nef positive part. We immediately obtain that $(X,B+M)$ has a log minimal model $(X',B'+M')$ according to the previous lemma. We can assume that the birational map $g:Y \dashrightarrow X'$ is a morphism, and it is obvious that $K_{X'}+B'+M'$ is semi-ample as $g^\ast (K_{X'}+B'+M')= P_\sigma (f^\ast(K_X+B+M))=P$ is semi-ample. .
	
	Next we prove the general case. Run an LMMP$/Z$ on $K_X+B+M$ with scaling of an ample$/Z$ divisor. From the argument above we reach a model $g:(X',B'+M') \dashrightarrow T$ after finitely many steps where $T$ is the lc model of $(X/Z,B+M)$ and $(K_{X'}+B'+M')|_{F_\eta} \sim_\Q 0$ on the generic fiber $F_\eta$ of $g$. Replacing $(X'/Z,B'+M')$ with some birational model, we can assume that $g$ is a morphism and write $K_{X'}+B'+M'\sim_\Q g^\ast A_T +N$ where $A_T$ is an ample$/Z$ divisor and $N \ge 0$ is vertical$/T$. By Lemma [\ref{B-lc-flips}, Lemma 3.2], $N$ is very exceptional$/T$, and hence by the proof of Lemma [\ref{B-lc-flips}, Theorem 3.4] any LMMP$/T$ on $K_{X'}+B'+M'$ with scaling of an ample$/Z$ divisor contracts $N$ and terminates with a minimal model $g'': (X'',B''+M'')\rightarrow T$ on which $K_{X''}+B''+ M'' \sim_\Q g''^\ast A_T$. 
\end{proof}

Next we prove a weak nonvanishing theorem via a similar argument from [\ref{BC}, Theorem 4.1].

\begin{thm}[Weak nonvanishing]\label{weaknonvanishing}
Let $(X,B+M)$ be a generalized lc generalized pair with data $X' \to X$ and $M'$, and $f:(X,B+M) \rightarrow Z$ be a surjective morphism to a normal variety $Z$. Assume that 
	
	$\bullet$ $Z$ has maximal Albanese dimension,

$\bullet$ $K_X+B+M$ is abundant$/Z$, 

$\bullet$ $\mathcal{R}(X/Z,K_X+B+M)$ is a finitely generated $\mathcal{O}_Z$-algebra, and

$\bullet$ $M'$ is semi-ample$/Z$. 

Then, there exists an effective divisor $D \ge 0$ such that $D \equiv K_X+B+M$.
\end{thm}

\begin{proof}
Replacing $(X,B+M)$ we can assume it is $\Q$-factorial generalized dlt polarized pair and there is a morphism $q: X' \rightarrow T$ over $Z$ such that $M' \sim_\Q q^\ast M_T$ for some relatively ample$/Z$ divisor $M_T$. Let $A=\mathrm{Alb}(Z)$. Pick an ample divisor $H$ on $A$ and a sufficiently small number $\delta >0$. Since $M_T$ is globally nef and big$/A$, it follows that $M'+\delta p'^\ast H$ is nef and abundant where $p':X' \rightarrow A$ is a morphism. Moreover, it is relatively semi-ample over a nonempty open subset $U \subset A$. 

Now run an LMMP$/Z$ on $K_X+B+M$ which terminates at a log minimal model $r:(X'',B''+M'') \rightarrow Y$ on which $K_{X''}+B''+M'' \sim_\Q r^\ast D_Y$ for some ample$/Z$ divisor $D_Y$ by Lemma \ref{lc-m-g-model}. Replacing $(X,B+M)$ with $(X'',B''+M'')$ and letting $P$ be the vertical$/Y$ part of $\lfloor B  \rfloor$ and $\epsilon>0$ be a sufficiently small rational number, we can assume $K_X+B+M$ is semi-ample$/Z$. As we argued in the proof oh Theorem \ref{m-thm'}, run an LMMP$/Y$ on $K_X+B-\epsilon P+ M$ which terminates at a good minimal model $X'' \to Y'/Y$. Again replacing $X$, $B$ and $Y$ with $X''$, $B''-\epsilon P''$ and $Y'$, we can assume that $P=0$. 

Let $S$ be a minimal lc center of $(X,B+M)$. Because every lc center of $(X,B+M)$ is horizontal$/Y$, $S$ is horizontal$/A$. We claim that there is an effective divisor $L$ such that $(X,B+L)$ is dlt, $(S,B_S+L_S)$ is klt and $K_X+B+L \sim_\Q K_X+B+M+\delta p^\ast H$  where $p:X \rightarrow A$ is a morphism. To this end, we pick sufficiently small rational numbers $\epsilon'$ and  $\delta'' \ll \delta$ so that $(1-\epsilon')M_T+\delta' t^*H=N+E$ where $t:T \to A$ is a morphism, $N$ is semi-ample and $E$ is effective with sufficiently small coefficients. If any sub-lc center of $(X',B')$ is contained in the support of $q^*E$, then it must be mapped to some horizontal$/A$ part of $E$. Therefore there is a divisor $E' \sim_{\mathbb{Q}} E+\epsilon'M_T$ such that the horizontal$/A$ part of $E'$ is effective and no sub-lc centers is contained in the support of $q^*E$. Note that $E'$ is not necessarily effective but we can suitably choose $H$ so that $E'+ (\delta-\delta')t^* H$ is effective which concludes the claim.

Applying a canonical bundle formula \ref{c-bdle-form-2} we get $D_Y +\delta g^\ast H \sim_\Q K_Y+\Delta_Y$ where $g: Y \rightarrow A$ is a morphism and $(Y,\Delta_Y)$ is klt. Now it suffices to prove the weak nonvanishing for $D_Y$. 

Replacing $(Y,\Delta)$ we can assume it is $\Q$-factorial. Since $K_Y+\Delta_Y$ is big$/A$, we can run an LMMP$/A$ on $K_Y+\Delta_Y$ which terminates at a good minimal model $(Y',\Delta_{Y'})$. In particular $K_{Y'}+\Delta_{Y'}$ is globally nef since $A$ does not contain any rational curve. Replacing $(Y,\Delta_Y)$ we can assume that $D_Y$ is nef. Replacing $(X,B+M)$ we can assume $D=K_X+B+M \sim_\Q r^\ast D_Y$. 

Let $I$ be a positive integer so that both $ID$ and $ID_Y$ are Cartier, and put $\mathscr{F}_s := g_\ast \mathcal{O}_Y(sID_Y)$. For any ample divisor $H$ on $A$ and any $P \in \mathrm{Pic}^0(A)$, by the Kawamata-Viehweg vanishing theorem,
$$
R^ig_\ast \mathcal{O}_Y(sID_Y +g^\ast H +g^\ast P)=0
$$
for all $i > 0$ and $s \ge 2$ and
$$
H^i(Y,sID_Y+g^\ast H +g^\ast P)=0
$$
for all $i > 0$ and $s \ge 2$, hence
$$
H^i(A,\mathscr{F}_s \otimes \mathcal{O}_A(H+P))
$$
for all $i > 0$ and $s \ge 2$.

We claim that $\mathscr{F}_s $ is a GV-sheaf. To this end, let $\phi_L: \widehat{A} \rightarrow A$ be the isogeny defined by a sufficiently positive ample line bundle $L$, one has $\phi_L^\ast (\widehat{L}^\vee) \cong \bigoplus\limits_{h^0(L)} L$. Let $\hat{g}: \widehat{Y}=Y \times_A \widehat{A} \rightarrow \widehat{A}$, $\varphi:\widehat{Y}\rightarrow Y $ be the induced morphism, and $\widehat{\mathscr{F}}_s=\hat{g}_\ast \mathcal{O}_{\widehat{Y}}(\varphi^\ast sID_Y)$. By Kawamata-Viehweg vanishing theorem, we have 
$$
H^i(\widehat{A},\widehat{\mathscr{F}}_s \otimes \mathcal{O}_{\widehat{A}}(L))=0
$$
for $i>0$, which in turn implies that
$$
H^i(A,\mathscr{F}_s \otimes L^\vee)=0
$$
for $i>0$.

Since $\mathscr{F}_s$ is a sheaf of rank $h^0(G, sID_Y|_G )$ where $G$ is a general fiber of $g$, it is a non-zero sheaf for $s$ sufficiently divisible. Therefore, $V^0(\mathscr{F}_s) \neq \emptyset$ otherwise $V^i(\mathscr{F}_s) = \emptyset$ for all $i$ which
implies that the Fourier-Mukai transform of $\mathscr{F}_s$ is zero. Pick a line bundle $P \in V^0(\mathscr{F}_s)$ and we conclude that $\kappa(sI(K_X+B+M)+f^\ast P) \ge 0$.
\end{proof}

\begin{rem}
Note that the assumption $\mathcal{R}(X/A,K_X+B+M)$ being finitely generated automatically holds when $(X,B+M)$ is generalized klt since abundance implies finite generation. We immediately obtain a corollary for generalized klt pairs. 
\end{rem}

\begin{cor}\label{nonvanishing-cor-1}
	Let $(X,B+M)$ be a generalized klt generalized pair with data $X' \to X$ and $M'$, and $f:(X,B+M) \rightarrow Z$ be a surjective morphism to a normal variety $Z$. Assume that 
	
	$\bullet$ $Z$ has maximal Albanese dimension,
	
	$\bullet$ $K_X+B+M$ is abundant$/Z$,  and
	
	$\bullet$ $M'$ is abundant$/Z$. 
	
	Then, there exists an effective divisor $D \ge 0$ such that $D \equiv K_X+B+M$.
\end{cor}

Since the relative abundance assumption is automatically satisfied when a fibration is relatively of Fano type, we immediately obtain the corollary below. Note that this includes Example \ref{exa1} and \ref{exa2}. 

\begin{cor}\label{nonvanishing-cor-2}
	Let $(X,B+M)$ be a generalized lc generalized pair with data $X' \to X$ and $M'$, and $f:(X,B+M) \rightarrow Z$ be a surjective morphism to a normal variety $Z$. Assume that 
	
	$\bullet$ $Z$ has maximal Albanese dimension,

$\bullet$ $\kappa(K_F+B_F+M_F) \ge 0$ where $F$ is a general fiber of $f$, and

$\bullet$ $X$ is relatively of log Fano type over $Z$, that is, there is a boundary $\Delta$ such that $(X,\Delta)$ is lc and $-(K_X+\Delta)$ is ample$/Z$.

Then, there exists an effective divisor $D \ge 0$ such that $D \equiv K_X+B+M$.
\end{cor}

\begin{rem}
One may ask whether the weak nonvanishing holds if we weaken the assumption of relative semi-ampleness of $M'$ in Theorem \ref{nonvanishing}. The answer is probably no. It should be interesting to know any counter example.
\end{rem}

\subsection{Variant: fibrations over irregular varieties}
From an easy observation one finds that the same argument still works if we slightly weaken the assumption that the base variety $Z$ has maximal Albanese dimension. 

We say that $Z$ is irregular \emph{with Albanese fiber of general type} if a general fiber of $a:Z \rightarrow A=\mathrm{Alb}(Z)$ is of general type.

The following lemma is well-known.
\begin{lem}[cf. \text{[\ref{HM}, Corollary 2.11]}]
Let $f:X \rightarrow Z$ be an algebraic fiber space, and let $(X,B)$ be an lc pair. Assume that $Z$ is of general type and $K_X+B$ is big over $Z$. Then, $K_X+B$ is big.
\end{lem}

\begin{proof}
Replacing $(X,B)$ and $Z$ we can assume that $(X,B)$ is log smooth and $Z$ is smooth. Thanks to [\ref{Nakayama}, V. 4.1 Theorem] we deduce that $\kappa_\sigma (K_X+B) \ge \kappa_\sigma(K_F+B|_F) +\kappa_\sigma (K_Z) =\dim X$ which in turn implies that $\kappa(K_X+B) =\dim X$ by [\ref{Nakayama}, V. 2.7 Proposition].
\end{proof}

\begin{prop}\label{prop-irr-1}
Let $f:(X,B) \rightarrow Z$ be a surjective morphism from a projective lc pair $(X,B)$ to a normal variety $Z$. Assume that $Z$ is irregular with Albanese fiber of general type and that $K_X+B$ is relatively abundant over $Z$. Then, $K_X+B$ is abundant. Moreover, the subadditvity of Kodaira dimensions $\kappa(K_X+B) \geq \kappa(K_F+B_F)+ \kappa(Z)$ holds, where $F$ is a general fiber, $K_F+B_F= (K_X+B)|_F$, and $\kappa(Z)$ means the Kodaira dimension of a smooth model of $Z$. 
\end{prop}

\begin{proof}
Let $Z \rightarrow Z' \rightarrow A$ be the Stein factorization. By assumption $K_Z$ is big over $Z'$. Let $g: X \dashrightarrow Y$ be a relative Iitaka fibration of $K_X+B$ over $Z$. Using a similar argument in the proof of Theorem \ref{m-thm} we get a generalized pair $(Y,B_Y+M_Y)$ such that $K_Y+B_Y+M_Y$ is big$/Z$, $\kappa(K_Y+B_Y+M_Y)=\kappa(K_X+B)$ and $\kappa_\sigma(K_Y+B_Y+M_Y)=\kappa_\sigma(K_X+B)$. If we modify the coefficients as in Remark \ref{rem-klt} so that $(Y,\Delta_Y)$ is klt and $K_Y+\Delta_Y \sim_\Q K_Y+B_Y^\epsilon+M_Y^\epsilon$, then we have that $K_Y+\Delta_Y$ is big over $Z'$ by the previous lemma. It follows that $\kappa(K_Y+\Delta_Y)=\kappa(K_X+B)$ and $\kappa_\sigma(K_Y+\Delta_Y)=\kappa_\sigma(K_X+B)$. The abundance of $K_X+B$ in turn holds as the abundance of $K_Y+\Delta_Y$ holds. By an easy calculation we also deduce the subadditivity of Kodaira dimensions.
\end{proof}

In the same way Proposition \ref{m-prop'} can be generalized as below.

\begin{prop}\label{prop-irr-2}
	Let $f:(X,B) \rightarrow Z$ be a surjective morphism from a projective lc pair $(X,B)$ to a normal variety $Z$. Assume that $Z$ is irregular with Albanese fiber of general type,  and that $(F,B_F)$ has a good minimal model where $F$ is the general fiber and $K_F+B_F=(K_X+B)|_F$. Then, $K_X+B$ is abundant. Moreover, if the canonical ring $\mathcal{R}(K_X+B)$ is finitely generated, for example, if $(X,B)$ is klt, or $\kappa(K_X+B)=0$, then $(X,B)$ has a good minimal model. 
\end{prop}

\vspace{0.3cm}
\section{A remark on Conjecture \ref{conj}}\label{rem-conj}

Finally we briefly discuss Conjecture \ref{conj} without the assumption of relative abundance. Let $f:(X,B) \rightarrow Z$ be a surjective morphism from a projective lc pair $(X,B)$ of dimension $d$ to a normal variety $Z$. Assume that $Z$ has maximal Albanese dimension. Replacing $(X,B)$ with a dlt blow-up, we can assume that $(X,B)$ is $\Q$-factorial dlt. 

\textbf{An inductive argument.}
Now we assume that Conjecture \ref{conj} holds in dimension $\le d-1$. Let $Z \rightarrow Z' \rightarrow A$ be the Stein factorization of its Albanese map. Following [\ref{Kawamata-2}, Theorem 13] there exist an abelian variety $B$, \'{e}tale covers $\widetilde{Z'}$ and $\widetilde{B}$ of $Z'$ and $B$ respectively, and a normal variety $\widetilde{C}$ such that \\
(1). $\widetilde{C}$ is finite over $C:=A/B$; \\
(2). $\widetilde{Z'}$ is isomorphic to $\widetilde{B}\times \widetilde{C}$; \\
(3). $\kappa(\widetilde{C})=\dim \widetilde{C} =\kappa(Z)$.\\
Let $\widetilde{X'}=X' \times_{Z'} \widetilde{Z'} $ and $\widetilde{Z}=Z \times_{Z'} \widetilde{Z'} $. $\widetilde{X'}$ and $\widetilde{Z}$ are \'{e}tale covers of $X'$ and $Z$ respectively. Replacing $X'$, $Z$ and $Z'$ with $\widetilde{X'}$, $\widetilde{Z}$ and $\widetilde{Z'}$, we can assume that $Z'=\widetilde{B} \times \widetilde{C}$ (see [\ref{BC}, Remark 5.1 and the proof of Theorem 1.1, Step 6] and [\ref{Kawamata-2}] for details).

If $\kappa(K_X+B)\geq 1$, then we are able to show the subadditivity of Kodaira dimensions 
$$
\kappa(K_X+B) \geq \kappa(K_F+B_F) +\kappa(Z)
$$
with an inductive argument as follows. Take a relative Iitaka fibration $g:X \dashrightarrow Y$ over $Z$, and an Iitaka fibration $h': Y \dashrightarrow T$ of $K_X+B$. We can assume $g$ and $h$ are morphisms and we have the commutative diagram.
$$
\xymatrix{
	X \ar[dd]_{f}\ar[dr]_{g}\ar[drr]^{h} & \\
	& Y  \ar[r]_{h'}\ar[dl]_{f'}  & T   \\
	Z  \ar[r]_{} & Z'=\widetilde{B}\times \widetilde{C} } 
$$

Let $G$ be a general fiber of $h$ and let $G'$ be the image of $G$ in $Y$. Because $\kappa(K_G+B_G)=0$, by an inductive assumption we deduce $\kappa(K_G+B_G/Z)=0$ and $\kappa(V)=0$ where $V$ is the normalization of the image of $G$ in $Z$. In particular, $G'$ is generically finite over $Z$, and $G$ is mapped to a point of $\widetilde{C}$ which in turn implies that $\dim G' \leq \dim Z' -\dim \widetilde{C} =\dim Z- \dim \widetilde{C} $. We therefore have 
\begin{align*}
\kappa(K_X+B) &=\dim X -\dim G \\
&=\dim Y -\dim G'\\
&\geq \dim Y - \dim Z +\dim \widetilde{C}  \\
&= \kappa(K_F+B_F) +\dim \kappa(Z)
\end{align*}

If we assume $\kappa(K_X+B)\leq 0$ and $\kappa(K_F+B_F) \geq 0$, then we deduce that $\kappa(K_X+B)=0$ by Lemma \ref{nonvanishing}. We can therefore reduce Conjecture \ref{conj} to the case when $\kappa(K_X+B)=0$. Let $f:X \rightarrow \widetilde{Z} \rightarrow Z$ be the Stein factorization. By [\ref{Kawamata-2}, Corollary 9] $\kappa(\widetilde{Z}) \geq \kappa(Z)$, hence it suffices to show that when $f$ is a surjective morphism with connected fibers if we replace $Z$ with $\widetilde{Z}$. As we argued before, we can assume that $Z'=\widetilde{B} \times \widetilde{C}$ where $\widetilde{C}$ is finite over an abelian variety and  $\kappa(\widetilde{C})=\dim \widetilde{C} =\kappa(Z)$. So, if $\kappa(K_X+B)=0$, then the image of $X$ in $\widetilde{C}$ is a point by [\ref{HM}, Corollary 2.11(2)] which in turn implies that $\kappa(Z)=0$. Because the main theorem of [\ref{Kawamata-2}] asserts that the Albanese map $\alpha: Z \rightarrow A$ is an algebraic fiber space, we are allowed to replace $Z$ with Albanese variety $A$. So, we can assume that $Z=A$ is an abelian variety.

\begin{rem}\label{remark}
	From the discussion above, we see that Conjecture \ref{conj} holds if it holds for any surjective morphism $f: (X,B) \rightarrow Z$ with connected fibers where $(X,B)$ is $\Q$-factorial dlt with $\kappa(K_X+B)=0$ and $Z=A$ is an abelian variety. In particular, Conjecture \ref{conj} holds for klt pairs by [\ref{CP}]  [\ref{HPS}].
\end{rem}

\textbf{Conjecture \ref{conj} with positivity on the boundary.}
As we mentioned, extending results from klt pairs to lc pairs is usually much more difficult than it sounds. Hence we consider adding an appropriate extra assumption. Note that if we suppose that $K_X+B$ is big$/Z$, then Conjecture \ref{conj} follows directly from Theorem \ref{m-thm}. So, we would like to see the feasibility of proving the conjecture if we put an extra bigness assumption on $B$. Unfortunately, the next proposition indicates that adding the bigness assumption will not decrease the difficulty of Conjecture \ref{conj}. 

To begin with, let us recall the global ACC Theorem [\ref{HMX}, Theorem D].

\begin{thm}\label{ACC-2}
	Fix a positive integer $n$ and a set $I \subset [0, 1]$, which satisfies the
	DCC.
	Then there is a finite set $I_0 \subset I$ with the following property: \\
	If $(X, \Delta)$ is an lc pair such that \\
	(i) $X$ is projective of dimension $n$, \\
	(ii) the coefficients of $\Delta$ belong to $I$, and \\
	(iii) $K_X + \Delta$ is numerically trivial, \\
	then the coefficients of $\Delta$ belong to $I_0$.
\end{thm}

\begin{prop}\label{relation-conj}
	The following statements are equivalent.
	
	(1). Conjecture \ref{conj} holds for lc pairs $(X,B)$ in dimension $\leq n$.
	
	(2). Conjecture \ref{conj} holds for $\mathbb{Q}$-factorial dlt pairs $(Y,B_Y)$ in dimension $\leq n+1$ where $B_Y$ is big$/Z$. 
	
	(3). Conjecture \ref{conj} holds for $\mathbb{Q}$-factorial dlt pairs $(Y,B_Y)$ in dimension $\leq n+1$ where $B_Y$ is big. 
\end{prop}

\begin{proof}
First we prove (1) implies (2). Assume that Conjecture \ref{conj} holds for lc pairs $(X,B)$ in dimension $\leq n-1$. Let $f:(Y,B_Y) \rightarrow Z$ be a surjective morphism from a projective dlt pair of dimension $n$ to a normal variety $Z$ with maximal Albanese dimension. In addition we suppose that $B_Y$ is big$/Z$.

Let $P_Y= \lfloor B_Y \rfloor$ and $A:=\mathrm{Alb}(Z)$. Pick a sufficiently small rational number $\epsilon >0$. Run an LMMP$/A$ with scaling of an ample divisor on $K_Y+B_Y-\epsilon P_Y$. If $K_Y+B_Y-\epsilon P_Y$ is pseudo-effective$/A$, then this LMMP ends up with a good minimal model $(Y',B_{Y'}-\epsilon P_{Y'})$ by [\ref{BCHM}] which is also a good log minimal model globally by Corollary \ref{m-cor'}. Moreover, by Lemma \ref{lem1} $\kappa(K_Y+B_Y)=\kappa(K_{Y'}+B_{Y'} -\epsilon P_{Y'})$ and $\kappa_\sigma(K_Y+B_Y)=\kappa_\sigma(K_{Y'}+B_{Y'} -\epsilon P_{Y'})$ provided $\epsilon$ sufficiently small. So, $K_Y+B_Y$ is abundant as $K_{Y'}+B_{Y'} -\epsilon P_{Y'}$ is abundant. In particular, since $\kappa(K_Y+B_Y/Z)=\kappa(K_Y+B_Y/A)=\kappa(K_{Y}+B_{Y} -\epsilon P_{Y}/Z)$, the subadditivity of Kodaira dimensions $\kappa(K_Y+B_Y) \geq \kappa(K_F+B_F)+ \kappa(Z)$ holds.

If $K_Y+B_Y-\epsilon P_Y$ is not pseudo-effective$/A$, then the LMMP above ends up with a Mori fiber space $g: (Y',B_{Y'}-\epsilon P_{Y'}) \rightarrow T$. If we choose $\epsilon$ sufficiently small, then by the global ACC of log canonical thresholds Theorem \ref{ACC-2} we have that $(Y',B_{Y'})$ is lc and $K_{Y'}+B_{Y'} \sim_\mathbb{Q} 0/T$. It follows that some component of $P_{Y'}$ is horizontal$/T$. Let $W$ be a common log resolution of $(Y,B_Y)$ and $(Y',B_{Y'})$. Write $p:W\rightarrow Y$ and $K_W+B_W=p^\ast(K_Y+B_Y)+E_W$ where $(W,B_W)$ is a log smooth model of $(Y,B_Y)$. Now we can run an LMMP$/Y'$ on $K_W+B_W$ which terminates at a good minimal model $(W',B_{W'})$ by [\ref{B-lc-flips}, Theorem 3.5, Theorem 1.1] because $(Y',B_{Y'})$ is $\mathbb{Q}$-factorial lc. More precisely, $K_W+B_W =q^\ast(K_{Y'}+B_{Y'}) +E_+ -E_-$ where $E_+$ and $E_-$ are effective $\mathbb{Q}$-divisors. Since an LMMP$/Y'$ contracts $E_+$ after finite steps, we have that $K_{W'}+B_{W'}+E'_- =q'^\ast(K_{Y'}+B_{Y'})$. 
$$
\xymatrix{
	W \ar[d]_{p}\ar@{-->}[r]^{}\ar[dr]_{q} &  W'\ar[d]^{q'} &\\
	Y  \ar@{-->}[r]^{}\ar[d]_{f}  & Y' \ar[r]^{g}\ar[d]_{}   & T \ar[dl]^{} \\
	Z \ar[r]^{} & A &} 
$$
Since $(W',B_{W'}+E'_-)$ is lc and $K_{W'}+B_{W'}+E'_- \sim_\mathbb{Q} 0 /T$, we can run an LMMP$/T$ on $K_{W'}+B_{W'}$ which terminates at a good minimal model $W'' \rightarrow T'$ by [\ref{B-lc-flips}]. As we pointed out earlier, some component of $P_{Y'}$ is horizontal$/T$. We denote its birational transform on $W$ by $S$ and its birational transform on $W''$ by $S''$. Pick a common resolution $\widetilde{S}$ of $S$ and $S''$. We obtain the following commutative diagram of projective morphisms. 
$$
\xymatrix{
	& \widetilde{S} \ar[dl] \ar[dr]&\\
	S \ar[d] \ar@{-->}[rr] &&  S'' \ar[d] \ar[dr]\\
	W \ar[d]_{p}\ar@{-->}[r]^{}\ar[dr]_{q} &  W'\ar[d]^{q'}\ar@{-->}[r] & W''\ar[d] \ar[r] & T' \ar[dl]^{\eta}\\
	Y  \ar@{-->}[r]^{}\ar[d]_{f}  & Y' \ar[r]^{g}\ar[d]_{}   & T \ar[dl]^{} \\
	Z \ar[r]^{} & A &} 
$$
Because $K_{W''}+B_{W''} \sim_\mathbb{Q} 0/T$, $\kappa(K_{W''}+B_{W''})=\kappa(K_{S''}+B_{S''})$ and $\kappa(K_{W''}+B_{W''}/A)=\kappa(K_{S''}+B_{S''}/A)$ where $K_{S''}+B_{S''}=(K_{W''}+B_{W''})|_{S''}$. Let $(\widetilde{S},B_{\widetilde{S}})$ be a log smooth model of $(S'',B_{S''})$. Then, $\kappa(K_{\widetilde{S}}+B_{\widetilde{S}})=\kappa(K_{S''}+B_{S''})$ and $\kappa(K_{\widetilde{S}}+B_{\widetilde{S}}/Z)=\kappa(K_{\widetilde{S}}+B_{\widetilde{S}}/A)=\kappa(K_{S''}+B_{S''}/A)$. By assumption we conclude the subadditivity of Kodaira dimensions.

Now we turn to proving that (3) implies (1) as the implication from (2) to (3) is obvious. Assume that Conjecture \ref{conj} holds for $\mathbb{Q}$-factorial dlt pairs $(Y,B_Y)$ in dimension $\leq n+1$ where $B_Y$ is big$/Z$. Given a surjective morphism $f:(X,B) \rightarrow Z$ from a projective lc pair of dimension $n$ to a normal variety $Z$ with maximal Albanese dimension. We construct a dlt pair $(Y,\Delta_Y)$ where $\Delta_Y$ is big$/Z$ as follows. Consider a $\mathbb{P}^1$-bundle 
$$
\pi: Y=\mathbb{P}(\mathcal{O}_X \oplus \mathcal{O}_X (-1)) \rightarrow X.
$$
Let $p: Y \rightarrow C=C(X)$ be the birational contraction of the negative section
$E$ on $Y$ and $H$ a general sufficiently ample $\mathbb{Q}$-divisor on $C$ such that $\lfloor H \rfloor = 0$ and $K_Y +E+B_Y+p^\ast H$ is big where $B_Y$ is the birational tranform of the cone of $B$ on $C$. Set 
$$
\lambda = \inf\{t| \text{ $K_Y +E+B_Y+ tp^\ast H$ is pseudo-effective$/X$}\}
$$
and $\Delta_Y=E+B_Y+ \lambda p^\ast H$. Obviously $\Delta_Y$ is big and $K_Y+\Delta_Y =\pi^\ast (K_X+B)$ since the induced morphism $E\rightarrow X$ is an isomorphism. This completes the argument.
\end{proof}

\vspace{0.3cm}

\vspace{2cm}

\flushleft{Center} of Mathematical Sciences,\\
Zhejiang University,\\
Yugu Road,\\
Hangzhou, 310000,\\
China.\\
email: zhengyuhu16@gmail.com

\vspace{1cm}

\end{document}